\documentclass[a4paper, 12pt, reqno]{amsart}

\usepackage{amssymb}
\usepackage[backref=page]{hyperref}
\usepackage[margin=0.75in]{geometry}
\usepackage{enumitem}
\usepackage{tikz-cd}
\usepackage{zref-clever}

\zcsetup{cap, nameinlink}

\newtheorem{theorem}{Theorem}[section]
\newtheorem{lemma}[theorem]{Lemma}
\AddToHook{env/lemma/begin}{\zcsetup{countertype={theorem=lemma}}}
\newtheorem{proposition}[theorem]{Proposition}
\AddToHook{env/proposition/begin}{\zcsetup{countertype={theorem=proposition}}}
\newtheorem{corollary}[theorem]{Corollary}
\AddToHook{env/corollary/begin}{\zcsetup{countertype={theorem=corollary}}}

\theoremstyle{remark}
\newtheorem{remark}[theorem]{Remark}
\AddToHook{env/remark/begin}{\zcsetup{countertype={theorem=remark}}}

\AddToHook{env/example/begin}{\zcsetup{countertype={theorem=example}}}

\zcRefTypeSetup{section}{
  Name-sg=\S,
  name-sg=\S,
  Name-pl=\S\S,
  name-pl=\S\S,
  namesep={}
}

\setlist[enumerate, 1]{label=\normalfont(\roman*)}

\DeclareMathOperator{\gr}{gr}
\DeclareMathOperator{\End}{End}
\DeclareMathOperator{\ch}{ch}
\DeclareMathOperator{\lm}{lm}
\DeclareMathOperator{\vspan}{span}
\DeclareMathOperator{\Ind}{Ind}
\DeclareMathOperator{\len}{len}

\newcommand{\vac}{{|0\rangle}}
\newcommand{\vachalf}{{|1/2\rangle}}
\newcommand{\vacsixteen}{{|1/16\rangle}}
\newcommand{\Vir}{{\mathrm{Vir}}}
\newcommand{\Id}{{\mathrm{Id}}}
\newcommand{\psn}{{\mathrm{psn}}}

\renewcommand*{\backref}[1]{}
\renewcommand*{\backrefalt}[4]{%
  \ifcase #1 (Not cited.)%
  \or        (Cited on page~#2.)%
  \else      (Cited on pages~#2.)%
  \fi}

\begin{document}

\setcounter{section}{-1}

\begin{abstract}
  To every $h + \mathbb{N}$-graded module $M$ over an $\mathbb{N}$-graded conformal vertex algebra $V$, we associate an increasing filtration $(G^pM)_{p \in \mathbb{Z}}$, which is compatible with the filtrations introduced by Haisheng Li.
  The associated graded vector space $\gr^G(M)$ is naturally a module over the vertex Poisson algebra $\gr^G(V)$.
  We study $\gr^G(M)$ for the three irreducible modules over the Ising model $\Vir_{3, 4}$, namely $\Vir_{3,4} = L(1/2, 0)$, $L(1/2, 1/2)$ and $L(1/2, 1/16)$.
  We obtain an explicit PBW basis of each of these modules and a formula for their refined characters, which are related to Nahm sums for the matrix $\left(\begin{smallmatrix} 8 & 3 \\ 3 & 2 \end{smallmatrix}\right)$. \\
\end{abstract}

\title{PBW bases of irreducible Ising modules}
\author{Diego Salazar}
\address{Instituto de Matemática Pura e Aplicada, Rio de Janeiro, RJ, Brazil}
\email{diego.salazar@impa.br}
\date{\today}
\maketitle

\section{Introduction}
\label{sec:introduction}

In \cite{li_vertex_2004}, Li introduced an increasing filtration $(G^pV)_{p \in \mathbb{Z}}$ on an arbitrary $\mathbb{N}$-graded vertex algebra $V$.
The associated graded space $\gr^G(V)$ with respect to this increasing filtration then carries the structure of an $\mathbb{N}$-graded vertex Poisson algebra.

Then in \cite{li_abelianizing_2005}, Li introduced a decreasing filtration $(F_pV)_{p \in \mathbb{Z}}$ on an arbitrary vertex algebra $V$, not necessarily $\mathbb{N}$-graded.
The associated graded space $\gr_F(V)$ with respect to this decreasing filtration again carries the structure of a vertex Poisson algebra.
Li also introduced a decreasing filtration $(F_pM)_{p \in \mathbb{Z}}$ for modules $M$ over a vertex algebra $V$ and showed that the associated graded space $\gr_F(M)$ is a module over the vertex Poisson algebra $\gr_F(V)$.

In summary, Li constructed three functors:
\begin{align*}
  \gr^G: \{\text{$\mathbb{N}$-graded vertex algebras}\} &\to \{\text{$\mathbb{N}$-graded vertex Poisson algebras}\}, \\
  \gr_F: \{\text{vertex algebras}\} &\to \{\text{vertex Poisson algebras}\}, \\
  \gr_F: \{\text{$V$-modules}\} &\to \{\text{$\gr_F(V)$-modules}\}.
\end{align*}

Then Arakawa showed in \cite[Proposition 2.6.1]{arakawa_remark_2012} that when $V$ is $\mathbb{N}$-graded, $\gr_F(V)$ and $\gr^G(V)$ are isomorphic as vertex Poisson algebras.

In this article, we define an increasing filtration $(G^pM)_{p \in \mathbb{Z}}$ for $h + \mathbb{N}$-graded modules $M$ over an $\mathbb{N}$-graded conformal vertex algebra $(V, \omega)$.
We construct a functor
\begin{equation*}
  \gr^G: \{\text{$h + \mathbb{N}$-graded $(V, \omega)$-modules}\} \to \{\text{$h + \mathbb{N}$-graded $\gr^G(V)$-modules}\}.
\end{equation*}
Again, the resulting modules $\gr^G(M)$ and $\gr_F(M)$ are isomorphic.
However, for our purposes, the filtration $(G^pM)_{p \in \mathbb{Z}}$ is better suited.

In \cite{andrews_singular_2022}, two theorems about the Virasoro minimal model $\Vir_{3, 4} = L(1/2, 0)$, also known as the Ising model, are proved.

\begin{theorem}
  \label{thr:1}
  The refined character of $\gr^G(\Vir_{3,4})$ is given by
  \begin{equation*}
    \ch_{\gr^G(\Vir_{3, 4})}(t, q) = \sum_{k_1, k_2 \in \mathbb{N}}t^{4k_1 + 2k_2}\frac{q^{4k_1^2 + 3k_1k_2 + k_2^2}}{(q)_{k_1}(q)_{k_2}}(1 - q^{k_1} + q^{k_1 + k_2}).
  \end{equation*}
\end{theorem}

Let $R^0$ be the following set of partitions
\begin{align*}
  &[r, r, r], [r + 1, r, r], [r + 1, r + 1, r], [r + 2, r + 1, r], [r + 2, r + 2, r], &(r \ge 2) \\
  &[r + 2, r, r], &(r \ge 3) \\
  &[r + 3, r + 3, r, r], [r + 4, r + 3, r, r],  [r + 4, r + 3, r + 1, r], [r + 4, r + 4, r + 1, r], &(r \ge 2) \\
  &[r + 6, r + 5, r + 3, r + 1, r], &(r \ge 2) \\
  &[5, 4, 2, 2], [7, 6, 4, 2, 2], [7, 7, 4, 2, 2], [9, 8, 6, 4, 2, 2].
\end{align*}
Let $P^0$ be the set of partitions $\lambda = [\lambda_1, \dots, \lambda_m]$ with $\lambda_m \ge 2$ that do not contain any partition in $R^0$.

\begin{theorem}
  \label{thr:2}
  The set
  \begin{equation*}
    \{L_{-\lambda_1}L_{-\lambda_2}\dots L_{-\lambda_m}\vac \mid \lambda = [\lambda_1, \dots, \lambda_m] \in P^0\}
  \end{equation*}
  is a vector space basis of $\Vir_{3, 4}$.
\end{theorem}

In this article, we prove two similar results about $L(1/2, 1/2)$ and $L(1/2, 1/16)$.

\begin{theorem}
  \label{thr:3}
  The refined character of $\gr^G(L(1/2, 1/2))$ is given by
  \begin{equation*}
    \ch_{\gr^G(L(1/2, 1/2))}(t, q) = q^{1/2}\left(\sum_{k_1, k_2 \in \mathbb{N}}t^{4k_1 + 2k_2}\frac{q^{4k_1^2 + 3k_1k_2 + k_2^2}}{(q)_{k_1}(q)_{k_2}}(q^{3k_1 + 2k_2} + q^{5k_1 + 2k_2 + 1}t + q^{6k_1 + 3k_2 + 2}t^2)\right).
  \end{equation*}
\end{theorem}

Let $R^{1/2}$ be the following set of partitions
\begin{align*}
  &[r, r, r], [r + 1, r, r], [r + 1, r + 1, r], [r + 2, r + 1, r], [r + 2, r + 2, r], &(r \ge 3) \\
  &[r + 2, r, r], &(r \ge 3) \\
  &[r + 3, r + 3, r, r], [r + 4, r + 3, r, r],  [r + 4, r + 3, r + 1, r], [r + 4, r + 4, r + 1, r], &(r \ge 3)\\
  &[r + 6, r + 5, r + 3, r + 1, r], &(r \ge 3) \\
  &[2], [1, 1, 1], [3, 1, 1], [3, 3], [4, 3, 1], [4, 4, 1], [5, 4, 1, 1], [6, 5, 3, 1].
\end{align*}
Let $P^{1/2}$ be the set of partitions that do not contain any partition in $R^{1/2}$.

\begin{theorem}
  \label{thr:4}
  The set
  \begin{equation*}
    \{L_{-\lambda_1}L_{-\lambda_2}\dots L_{-\lambda_m}\vachalf \mid \lambda = [\lambda_1, \dots, \lambda_m] \in P^{1/2}\}
  \end{equation*}
  is a vector space basis of $L(1/2, 1/2)$.
\end{theorem}

Specializing the character formula in \zcref{thr:3} to $t = 1$, we obtain that the character of $L(1/2, 1/2)$ is the sum of three Nahm sums for the same matrix $\left(\begin{smallmatrix} 8 & 3 \\ 3 & 2 \end{smallmatrix}\right)$ (cf.\ \cite{Nahm2007} and \cite{andrews_singular_2022}).
The partitions in $P^0$ also have a combinatorial interpretation, as was noted in \cite{andrews_singular_2022} and \cite{tsuchioka_vertex_2023}.
For example, for $n \in \mathbb{N}$, the number of partitions of $n$ in $P^0$ is the number of partitions of $n$ with parts congruent to $\pm2$, $\pm3$, $\pm4$ and $\pm5$ modulo $16$.

\begin{theorem}
  \label{thr:5}
  The refined character of $\gr^G(L(1/2, 1/16))$ is given by
  \begin{equation*}
    \ch_{\gr^G(L(1/2, 1/16))}(t, q) = q^{1/16}\left(\sum_{k_1, k_2 \in \mathbb{N}}t^{4k_1 + 2k_2}\frac{q^{4k_1^2 + 3k_1k_2 + k_2^2}}{(q)_{k_1}(q)_{k_2}}(q^{k_1 + k_2} + q^{4k_1 + 2k_2 + 1}t + q^{7k_1 + 3k_2 + 3}t^3)\right).
  \end{equation*}
  Let $R^{1/16}$ be the following set of partitions
  \begin{align*}
    &[r, r, r], [r + 1, r, r], [r + 1, r + 1, r], [r + 2, r + 1, r], [r + 2, r + 2, r], &(r \ge 3) \\
    &[r + 2, r, r], &(r \ge 3) \\
    &[r + 3, r + 3, r, r], [r + 4, r + 3, r, r],  [r + 4, r + 3, r + 1, r], [r + 4, r + 4, r + 1, r], &(r \ge 3)\\
    &[r + 6, r + 5, r + 3, r + 1, r], &(r \ge 3) \\
    &[2], [1, 1, 1, 1], [3, 1, 1, 1], [3, 3, 1], [4, 3, 1], [4, 4, 1, 1], [5, 4, 1, 1, 1], [5, 5, 1, 1, 1], \\
    &[6, 5, 3, 1, 1], [6, 6, 3, 1, 1], [7, 6, 4, 1, 1, 1], [8, 7, 5, 3, 1, 1].
  \end{align*}
  Let $P^{1/16}$ be the set of partitions that do not contain any partition in $R^{1/16}$.
  The set
  \begin{equation*}
    \{L_{-\lambda_1}L_{-\lambda_2}\dots L_{-\lambda_m}\vacsixteen \mid \lambda = [\lambda_1, \dots, \lambda_m] \in P^{1/16}\}
  \end{equation*}
  is a vector space basis of $L(1/2, 1/16)$.
\end{theorem}

This article is organized as follows.
In \zcref{sec:prel-notat}, we review Verma modules, the standard filtration, Gröbner bases and consider Nahm-like sums of two variables.
In \zcref{sec:comb-argum}, we compute the series $p(t, q)$ related to $P^{1/2}$ using combinatorial arguments.
In \zcref{sec:comp-lead-monom}, we exhibit the partitions in $P^{1/2}$ using Wang's results in \cite{wang_rationality_1993} and explain how we use software to get the exceptional partitions.
In \zcref{sec:proofs-main-theorems}, we prove \zcref{thr:3} and \zcref{thr:4}.
In the appendices, we show the SageMath \cite{sagemath} program used to compute the exceptional partitions appearing in $R^{1/2}$, and we briefly consider the case $L(1/2, 1/16)$.
The software systems Mathematica \cite{Mathematica} and Singular \cite{Singular} were also very useful to verify the series and compute Gröbner bases.

I would like to thank my advisor Reimundo Heluani and Instituto de Matemática Pura e Aplicada (IMPA) for their support.
Jethro Van Ekeren made some valuable suggestions as well.
The author is partially supported by PhD scholarship 155672/2019-3 from Conselho Nacional de Desenvolvimento Científico e Tecnológico (CNPq).

\section{Preliminaries and notation}
\label{sec:prel-notat}

The set of natural numbers $\{0, 1, \dots\}$ is denoted by $\mathbb{N}$, the set of integers is denoted by $\mathbb{Z}$, and the set of positive integers $\{1, 2, \dots\}$ is denoted by $\mathbb{Z}_+$.

First, we review the theory of representations of the Virasoro Lie algebra following \cite{kac_bombay_2013}.
The \emph{Virasoro Lie algebra} is a Lie algebra given by
\begin{equation*}
  \Vir = \bigoplus_{n \in \mathbb{Z}}\mathbb{C}L_n \oplus \mathbb{C}C.
\end{equation*}
These elements satisfy the following commutation relations:
\begin{align*}
  [L_m, L_n] &= (m - n)L_{m + n} + \delta_{m, -n}\frac{m^3 - m}{12}C \quad \text{for $m, n \in \mathbb{Z}$}, \\
  [\Vir, C] &= 0.
\end{align*}
Let $c, h \in \mathbb{C}$.
We set $\Vir_{\ge 0} = \bigoplus_{n \in \mathbb{N}}\mathbb{C}L_n$.
We make the subalgebra $\Vir_{\ge 0} \oplus \mathbb{C}C$ of $\Vir$ act on $\mathbb{C}$ as follows:
\begin{equation*}
  \text{$L_n1 = 0$ for $n \in \mathbb{Z}_+$, $L_01 = h$ and $C1 = c$}.
\end{equation*}
The \emph{Verma representation} of $\Vir$ with \emph{highest weight} $(c, h)$ is defined as
\begin{equation*}
  M(c, h) = \Ind^{\Vir}_{\Vir_{\ge 0} \oplus \mathbb{C}C}(\mathbb{C}) = U(\Vir) \otimes_{U(\Vir_{\ge 0} \oplus \mathbb{C}C)} \mathbb{C},
\end{equation*}
where $\Vir$ acts by left multiplication.
We take $|c, h\rangle = 1\otimes1$ as \emph{highest weight vector}.

A \emph{partition (of $n \in \mathbb{N}$)} is a sequence $\lambda = [\lambda_1, \dots, \lambda_m]$ such that $\lambda_i \in \mathbb{Z}_+$ for $i = 1, \dots, m$, $\lambda_1 \ge \dots \ge \lambda_m$ (and $\lambda_1 + \dots + \lambda_m = n$).
We also consider the \emph{empty partition} $\emptyset$, which is the unique partition of $0$.
For a partition $\lambda = [\lambda_1, \dots, \lambda_m]$, we define
\begin{equation*}
  L_{\lambda} = L_{-\lambda_1}\dots L_{-\lambda_m} \in U(\Vir).
\end{equation*}
By the PBW theorem, the set
\begin{equation*}
  \{L_{\lambda}|c, h\rangle \mid \text{$\lambda$ is a partition}\}
\end{equation*}
is a vector space basis of $M(c, h)$.
The representation $M(c, h)$ has a unique maximal proper subrepresentation $J(c, h)$, and the quotient
\begin{equation*}
  L(c, h) = M(c, h)/J(c, h)
\end{equation*}
is the \emph{irreducible highest weight representation} of $\Vir$ with highest weight $(c, h)$.
We usually simplify $|c, h\rangle$ to just $|h\rangle$ when $c$ is understood.
In \cite{andrews_singular_2022}, $\Vir_{3, 4} = L(1/2, 0)$ was studied.
In this article, we study $L(1/2, 1/2)$ and $L(1/2, 1/16)$.
These three representations are the only irreducible modules of $\Vir_{3, 4}$ as a vertex algebra (see \cite[Theorem 4.2]{wang_rationality_1993}).

Let $V$ be an $\mathbb{N}$-graded vertex algebra, meaning it has a Hamiltonian $H \in \End(V)$ whose eigenvalues are natural numbers (see \cite[\S1.6]{de_sole_finite_2006}).
Let $(a^i)_{i \in I}$ be a family of homogeneous strong generators of $V$.
For $p \in \mathbb{Z}$, we set
\begin{equation*}
  G^pV = \vspan\{a^{i_1}_{(-n_1 - 1)}\dots a^{i_s}_{(-n_s - 1)}\vac \mid s, n_1, \dots, n_s \in \mathbb{N}, i_1, \dots, i_s \in I, \Delta_{a^{i_1}} + \dots + \Delta_{a^{i_s}} \le p\},
\end{equation*}
where $\Delta$ denotes conformal weight.

\begin{proposition}[{\cite{li_vertex_2004}}]
  \label{prp:1}
  The filtration $(G^pV)_{p \in \mathbb{Z}}$ satisfies:
  \begin{enumerate}
  \item $G^pV = 0$ for $p < 0$;
  \item $\vac \in G^0V \subseteq G^1V \subseteq \dots$;
  \item $V_n \subseteq G^nV$ for $n \in \mathbb{Z}$;
  \item $V = \bigcup_{p \in \mathbb{N}}G^pV$;
  \item $a_{(n)}G^qV \subseteq G^{p + q}V$ for $p, q \in \mathbb{Z}$, $a \in G^pV$ and $n \in \mathbb{Z}$;
  \item $a_{(n)}G^qV \subseteq G^{p + q - 1}V$ for $p, q \in \mathbb{Z}$, $a \in G^pV$ and $n \in \mathbb{N}$;
  \item $H(G^pV) \subseteq G^pV$ and $T(G^pV) \subseteq G^pV$ for $p \in \mathbb{Z}$.
  \end{enumerate}
\end{proposition}

Let
\begin{equation*}
  \gr^G(V) = \bigoplus_{p \in \mathbb{N}}G^pV/G^{p - 1}V
\end{equation*}
be the associated graded vector space.
By \cite{li_vertex_2004}, the vector space $\gr^G(V)$ is a vertex Poisson algebra with operations given as follows.
For $p, q \in \mathbb{N}$, $a \in G^pV$ and $b \in G^qV$, we set:
\begin{align*}
  \alpha^p(a)\alpha^q(b) &= \alpha^{p + q}(a_{(-1)}b), \\
  T\alpha^p(a) &= \alpha^p(Ta), \\
  Y_-(\alpha^p(a), z)\alpha^q(b) &= \sum_{n \in \mathbb{N}}\alpha^{p + q - 1}(a_{(n)}b)z^{-n - 1},
\end{align*}
where $\alpha^p: G^pV \to \gr^G(V)$ is the \emph{principal symbol map}.
The unit is $\alpha^0(\vac)$.
The filtration $(G^pV)_{p \in \mathbb{Z}}$ is called the \emph{standard filtration of $V$}.

Let $(V, \omega)$ be an $\mathbb{N}$-graded conformal vertex algebra, let $(a^i)_{i \in I}$ be a family of homogeneous strong generators of $V$, and let $M$ be an $h + \mathbb{N}$-graded $(V, \omega)$-module.
As usual, we write $Y^M(\omega, z) = \sum_{n \in \mathbb{Z}}L^M_nz^{-n - 2}$.
This means $M$ is a $V$-module with $L_0^M$ diagonalizable whose eigenvalues are in the set $h + \mathbb{N}$ for some $h \in \mathbb{C}$.
We set $M_{\Delta} = \ker(L^M_0 - \Delta\Id_M)$ for $\Delta \in \mathbb{C}$, so we have $M = \bigoplus_{n \in \mathbb{N}}M_{h + n}$.
For $p \in \mathbb{Z}$, we set
\begin{equation*}
  \begin{split}
    G^pM = \vspan\{a^{i_1M}_{(-n_1 - 1)}\dots a^{i_sM}_{(-n_s - 1)}u &\mid \text{$s, n_1, \dots, n_s \in \mathbb{N}$, $i_1, \dots, i_s \in I$, $u \in M$ homogeneous,} \\
                                                                     &\quad \Delta_{a^{i_1}} + \dots + \Delta_{a^{i_s}} + \Delta_u - h \le p\}.
  \end{split}
\end{equation*}

\begin{proposition}
  \label{prp:2}
  The filtration $(G^pM)_{p \in \mathbb{Z}}$ satisfies:
  \begin{enumerate}
  \item $G^pM = 0$ for $p < 0$;
  \item $G^0M \subseteq G^1M \subseteq \dots$;
  \item $M = \bigcup_{p \in \mathbb{N}}G^pM$;
  \item $M_{h + n} \subseteq G^nM$ for $n \in \mathbb{Z}$;
  \item $a^M_{(n)}G^qM \subseteq G^{p + q}M$ for $p, q \in \mathbb{Z}$, $a \in G^pV$ and $n \in \mathbb{Z}$;
  \item $a^M_{(n)}G^qM \subseteq G^{p + q - 1}M$ for $p, q \in \mathbb{Z}$, $a \in G^pV$ and $n \in \mathbb{N}$;
  \item $L^M_0(G^pM) \subseteq G^pM$ and $L^M_{-1}(G^pM) \subseteq G^{p + 1}M$ for $p \in \mathbb{Z}$.
  \end{enumerate}
\end{proposition}

\begin{proof}
  The proofs in \cite{li_vertex_2004} also work here.
\end{proof}

\begin{remark}
  \label{rmk:1}
  We do not have the property $L_{-1}^M(G^pM) \subseteq G^pM$ as in the case of vertex algebras.
\end{remark}

Let
\begin{equation*}
  \gr^G(M) = \bigoplus_{p \in \mathbb{N}}G^pM/G^{p - 1}M
\end{equation*}
be the associated graded vector space.
The vector space $\gr^G(M)$ is a module over $\gr^G(V)$ with operations given as follows.
For $p, q \in \mathbb{N}$, $a \in G^pV$ and $u \in G^qM$, we set:
\begin{align*}
  \alpha^p(a)\alpha_M^q(u) &= \alpha_M^{p + q}(a^M_{(-1)}u), \\
  Y^M_-(\alpha^p(a), z)\alpha_M^q(u) &= \sum_{n \in \mathbb{N}}\alpha_M^{p + q - 1}(a^M_{(n)}u)z^{-n - 1},
\end{align*}
where $\alpha_M^p: G^pM \to \gr^G(M)$ is the principal symbol map.
The filtration $(G^pM)_{p \in \mathbb{Z}}$ is called the standard filtration of $M$.

By \zcref{prp:2}(vii), we can define a diagonalizable operator $H^M \in \End(\gr^G(M))$ by setting $H^M(\alpha^p_M(u)) = \alpha^p_M(L^M_0u)$ for $p \in \mathbb{N}$ and $u \in G^pM$.

For $p, n \in \mathbb{N}$, we set $G^pM_{h + n} = G^pM \cap M_{h + n}$.
Then
\begin{equation}
  \label{eq:1}
  \gr^G(M) = \bigoplus_{n \in \mathbb{N}}\gr^G(M)_{h + n},
\end{equation}
where $\gr^G(M)_{h + n} = \bigoplus_{p \in \mathbb{N}}\alpha^p_M(G^pM_{h + n})$ and $\gr^G(M)_{h + n} = \ker(H^M - (h + n)\Id_{\gr^G(M)})$.
Besides the grading \eqref{eq:1}, we have the refined grading
\begin{equation}
  \label{eq:2}
  \gr^G(M) = \bigoplus_{p, n \in \mathbb{N}}\alpha^p_M(G^pM_{h + n}).
\end{equation}
We assume that for $n \in \mathbb{N}$, $M_{h + n}$ is finite dimensional.
We define the \emph{character} and the \emph{refined character of $\gr^G(M)$} (with respect to grading \eqref{eq:1} and grading \eqref{eq:2}) as:
\begin{align*}
  \ch_{\gr^G(M)}(q) &= \sum_{n \in \mathbb{N}}\dim(\gr^G(M)_{h + n})q^{h + n} \in q^h\mathbb{C}[[q]], \\
  \ch_{\gr^G(M)}(t, q) &= \sum_{p, n \in \mathbb{N}}\dim(\alpha^p_M(G^pM_{h + n}))t^pq^{h + n} \in q^{h}\mathbb{C}[[t, q]].
\end{align*}
By \zcref{prp:2}, we have
\begin{equation*}
  \ch_M(q) = \ch_{\gr^G(M)}(q) = \ch_{\gr^G(M)}(1, q).
\end{equation*}

If $f: M_1 \to M_2$ is a homomorphism of $h + \mathbb{N}$-graded $(V, \omega)$-modules, then
\begin{align*}
  \gr^G(f): \gr^G(M_1) &\to \gr^G(M_2), \\
  \gr^G(f)(\alpha^p_{M_1}(u)) &= \alpha^p_{M_2}(f(u)) \quad \text{for $p \in \mathbb{N}$ and $u \in G^pM_1$}
\end{align*}
defines a homomorphism of $h + \mathbb{N}$-graded $\gr^G(V)$-modules.
Therefore, we obtain a functor
\begin{equation*}
  \gr^G: \{\text{$h + \mathbb{N}$-graded $(V, \omega)$-modules}\} \to \{\text{$h + \mathbb{N}$-graded $\gr^G(V)$-modules}\}.
\end{equation*}

We now state a general linear algebra proposition about bases of increasing filtrations of vector spaces.
The proposition applies in particular to the standard filtration $(G^pM)_{p \in \mathbb{Z}}$ of a module $M$ over a vertex algebra, and the proof is straightforward.

\begin{proposition}
  \label{prp:3}
  Let $M$ be a vector space, let $(M^p)_{p \in \mathbb{Z}}$ be a family of subspaces of $M$ satisfying properties \emph{(i)--(iii)} of \zcref{prp:2}, and let $(B^p)_{p \in \mathbb{N}}$ be an increasing family of sets satisfying $B^p \subseteq M^p$ for $p \in \mathbb{N}$.
  Then $\{v + M^{p - 1} \mid v \in B^p \setminus B^{p - 1}\}$ spans (resp.\ is a basis of) $M^p/M^{p - 1}$ for $p \in \mathbb{N}$ if and only if $B^p$ spans (resp.\ is a basis of) $M^p$ for $p \in \mathbb{N}$ ($B^{-1} = \emptyset$).
  In particular, if $\{v + M^{p - 1} \mid p \in \mathbb{N}, v \in B^p \setminus B^{p - 1}\}$ is a basis of $\gr(M) = \bigoplus_{p \in \mathbb{N}}M^p/M^{p - 1}$, then $\bigcup_{p \in \mathbb{N}}B^p$ is a basis of $M$.\footnote{The version published in \cite{salazar_pbw_2024} has a minor mistake in Proposition 1.8 because it presents an explicit isomorphism $L \xrightarrow{\sim} \gr^G(L), L_{\lambda}(\vachalf + W) \mapsto \alpha^{\len(\lambda)}(L_{\lambda}(\vachalf + W))$ which is not well-defined.
    For example, using the notation in \zcref{sec:prel-notat} and \zcref{sec:an-expl-descr} ahead, $L_{-1}^3 = 0$ in $\gr^G(L)$ but $L_{-1}^3\vachalf \notin W$.
    This proposition fixes that mistake.
    Ultimately, we only use that $\dim(L_{h + n}) = \dim(\gr^G(L)_{h + n})$ for $n \in \mathbb{N}$, so the remaining results are not affected.}
\end{proposition}

\begin{proposition}
  \label{prp:4}
  Let $(V, \omega)$ be an $\mathbb{N}$-graded conformal vertex algebra, and let $M$ be an $h + \mathbb{N}$-graded $(V, \omega)$-module.
  Then the Li filtration (see \cite{li_abelianizing_2005}) and the standard filtration satisfy
  \begin{equation*}
    F_pM_{h + n} = G^{n - p}M_{h + n} \quad \text{for $p, n \in \mathbb{N}$},
  \end{equation*}
  which implies $\gr_F(M)$ and $\gr^G(M)$ are isomorphic as modules.
\end{proposition}

\begin{proof}
  The proof is essentially the same as the proof of \cite[Proposition 2.6.1]{arakawa_remark_2012}.
\end{proof}

\begin{remark}
  \label{rmk:2}
  The definition of the standard filtration $(G^pM)_{p \in \mathbb{Z}}$ apparently depends on $h$, so we should write it as $(G^p_hM)_{p \in \mathbb{Z}}$.
  However, since $G^p_{h - 1}M = G^{p - 1}_hM$ for $p \in \mathbb{Z}$, $\gr^G_h(M)$ and $\gr^G_{h - 1}(M)$ are isomorphic and $\ch_{\gr^G_{h - 1}(M)}(t, q) = t\ch_{\gr^G_h(M)}(t, q)$.
  If we further require that $M_h \neq 0$ when $M \neq 0$, then $\gr^G(M)$ and $\ch_{\gr^G(M)}(t, q)$ are well-defined.
\end{remark}

\begin{remark}
  \label{rmk:3}
  We can relax the hypothesis of $L_0^M$ being diagonalizable.
  It is possible to define the standard filtration for admissible modules (see \cite[Definition 3.3]{dong_twisted_1998}), and \zcref{prp:4} still holds.
\end{remark}

From now on, we take $V = \Vir^{1/2}$ (the universal Virasoro vertex algebra of central charge $1/2$) as the $\mathbb{N}$-graded vertex algebra with conformal vector $\omega = L_{-2}\vac$.
Some subscripts or superscripts will be omitted, so for example $\alpha^p_M$ simplifies to $\alpha^p$.
We set:
\begin{equation*}
  M = M(1/2, 1/2), L = L(1/2, 1/2), W = J(1/2, 1/2).
\end{equation*}
The representations $M$ and $L$ can be considered as modules over $V$ because they are smooth representations of $\Vir$ of central charge $1/2$ (cf.\ \cite[Theorem 6.1.7]{lepowsky_introduction_2004}).
By definition, we have $L = M/W$.
It can be shown (cf.\ \cite{astashkevich_structure_1997}) that
\begin{equation}
  \label{eq:3}
  W = U(\Vir)\{u_2, u_3\},
\end{equation}
where the generating singular vectors are:
\begin{equation}
  \label{eq:4}
  u_2 = (L_{-1}^2 - \tfrac{4}{3}L_{-2})\vachalf, u_3 = (L_{-1}^3 - 3L_{-2}L_{-1} + \tfrac{3}{4}L_{-3})\vachalf.
\end{equation}

We have a natural epimorphism of representations of $\Vir$
\begin{align*}
  \pi: M &\twoheadrightarrow L, \\
  \pi(u) &= u + W.
\end{align*}
We can consider $\pi$ as a homomorphism of modules over $V$.
Applying the functor $\gr^G$, we obtain an epimorphism of modules over $\gr^G(V)$
\begin{equation*}
  \gr^G(\pi): \gr^G(M) \twoheadrightarrow \gr^G(L),
\end{equation*}
and this produces a natural isomorphism of modules over $\gr^G(V)$
\begin{equation*}
  \gr^G(M)/K \xrightarrow{\sim} \gr^G(L),
\end{equation*}
where
\begin{equation}
  \label{eq:5}
  K = \ker(\gr^G(\pi)).
\end{equation}

Let us describe these objects and the isomorphisms explicitly.
For $V$, $\gr^G(V)$ is isomorphic to $\mathbb{C}[L_{-2}, L_{-3}, \dots]$ as a differential commutative associative algebra with the derivation given by
\begin{align*}
  \partial: \mathbb{C}[L_{-2}, L_{-3}, \dots] &\to \mathbb{C}[L_{-2}, L_{-3}, \dots], \\
  \partial(L_{-n}) &= (n - 1)L_{-n - 1} \quad \text{for $n \ge 2$}.
\end{align*}
The vertex Poisson structure (i.e., the map $Y_-$) is trivial, meaning it is $0$.
This is because
\begin{equation*}
  \gr^G(V) \cong \gr(U(\Vir_{\le -2})) \cong S(\Vir_{\le -2}) \cong \mathbb{C}[L_{-2}, L_{-3}, \dots],
\end{equation*}
where $\Vir_{\le -2}$ is the subalgebra of $\Vir$ given by $\Vir_{\le -2} = \bigoplus_{n \le -2}\mathbb{C}L_n$, and $U(\Vir_{\le -2})$ is equipped with the PBW filtration (see \cite[\S2]{dixmier_enveloping_1996} and \cite[(5.2.25)]{lepowsky_introduction_2004}).
A similar argument can be made for $\gr^G(M)$, obtaining the following isomorphism
\begin{align*}
  \gr^G(M) &\xrightarrow{\sim} \bigoplus_{k \in \mathbb{N}}\mathbb{C}[L_{-2}, L_{-3}, \dots]L_{-1}^k, \\
  \alpha_M^{2s + k}(L_{-n_1 - 2}^M\dots L_{-n_s - 2}^M(L_{-1}^M)^k\vachalf) &\mapsto L_{-n_1 - 2}\dots L_{-n_s - 2}L_{-1}^k,
\end{align*}
where $s, k, n_1, \dots, n_s \in \mathbb{N}$.

\begin{remark}
  \label{rmk:4}
  That $+k$ in the isomorphism above is what makes this filtration different from the PBW filtration, where all $L_n$ for $n \le -1$ have the same length.
  On the other hand, with the standard filtration, $L_{-1}$ has length equal to $1$, while $L_{-2}$, $L_{-3}$, $\dots$ have length equal to $2$.
\end{remark}

These observations justify the following definitions.
For a partition $\lambda = [\lambda_1, \dots, \lambda_m, 1, \dots, 1]$ with exactly $n$ ones, we define $u_{\lambda} \in \bigoplus_{k \in \mathbb{N}}\mathbb{C}[L_{-2}, L_{-3}, \dots]L_{-1}^k$ by setting
\begin{equation*}
  u_{\lambda} = L_{-\lambda_1}\dots L_{-\lambda_m}L_{-1}^n.
\end{equation*}
We define the \emph{length of $\lambda$} as
\begin{equation*}
  \len(\lambda) = 2m + n
\end{equation*}
and the \emph{weight of $\lambda$} as
\begin{equation*}
  \Delta(\lambda) = \lambda_1 + \dots + \lambda_m + n.
\end{equation*}
Therefore, for a partition $\lambda$, we have
\begin{equation*}
  L_{\lambda}\vachalf \in G^{\len(\lambda)}M_{1/2 + \Delta(\lambda)}.
\end{equation*}
For a partition $\lambda = [\lambda_1, \dots, \lambda_m]$ with $\lambda_m \ge 2$, we define $p_{\lambda} \in \mathbb{C}[L_{-2}, L_{-3}, \dots]$ by setting
\begin{equation*}
  p_{\lambda} = L_{-\lambda_1}\dots L_{-\lambda_m}.
\end{equation*}

In this article, we will deal with polynomial algebras written as $\mathbb{C}[L_{-2}, L_{-3}, \dots, L_{-N}]$ and free modules of the form $\bigoplus_{n \le N}\mathbb{C}[L_{-2}, L_{-3}, \dots, L_{-N}]L_{-1}^n$ for some $N \in \mathbb{N}$.
We will always use the degree reverse lexicographic order with $L_{-2} > L_{-3} > \dots > L_{-N}$ and TOP (term over position) with $L_{-1}^0 < L_{-1}^1 < \dots < L_{-1}^N$, see \cite[\S3.5]{adams_introduction_1994}.
Let $u \in \bigoplus_{k \in \mathbb{N}}\mathbb{C}[L_{-2}, L_{-3}, \dots]L_{-1}^k$.
We can define the \emph{leading monomial of $u$}, denoted by $\lm(u)$, as follows.
We pick $N$ large enough so that $u \in \bigoplus_{n \le N}\mathbb{C}[L_{-2}, L_{-3}, \dots, L_{-N}]L_{-1}^n$.
Then, we define $\lm(u)$ as the leading monomial in $\bigoplus_{n \le N}\mathbb{C}[L_{-2}, L_{-3}, \dots, L_{-N}]L_{-1}^n$, which is naturally a subset of $\bigoplus_{k \in \mathbb{N}}\mathbb{C}[L_{-2}, L_{-3}, \dots]L_{-1}^k$.
This does not depend on the choice of $N$.

When working with $q$-series, the following notation is useful.
The \emph{$q$-Pochhammer symbol} is $(q)_n = \prod_{j = 1}^n(1 - q^j) \in \mathbb{C}[q]$ for $n \in \mathbb{N}$.
The recursive proof of the series identities in \cite{andrews_singular_2022} suggests the following definition.
For $a, b, c, d \in \mathbb{N}$, we define
\begin{equation*}
  f_{a, b, c, d}(t, q) = \sum_{k_1, k_2 \in \mathbb{N}}t^{4k_1 + 2k_2 + d}\frac{q^{4k_1^2 + 3k_1k_2 + k_2^2 + ak_1 + bk_2 + c}}{(q)_{k_1}(q)_{k_2}} \in \mathbb{C}[[t, q]].
\end{equation*}
The following lemma will be used frequently to verify series identities.

\begin{lemma}
  \label{lmm:1}
  The series $f_{a, b, c, d}(t, q)$ satisfies:
  \begin{enumerate}
  \item $t^mq^nf_{a, b, c, d}(t, q) = f_{a, b, c + n, d + m}(t, q)$ for $m, n \in \mathbb{N}$;
  \item $f_{a, b, c, d}(tq^n, q) = f_{a + 4n, b + 2n, c + dn, d}(t, q)$ for $n \in \frac{1}{2}\mathbb{N}$ and $d \in 2\mathbb{N}$;
  \item $f_{a, b, c, d}(t, q) - f_{a + n, b, c, d}(t, q) = \sum_{k = 0}^{n - 1}f_{a + 8 + k, b + 3, a + c + 4 + k, d + 4}(t, q)$ for $n \in \mathbb{Z}_+$;
  \item $f_{a, b, c, d}(t, q) - f_{a, b + n, c, d}(t, q) = \sum_{k = 0}^{n - 1}f_{a + 3, b + 2 + k, b + c + 1 + k, d + 2}(t, q)$ for $n \in \mathbb{Z}_+$.
  \end{enumerate}
\end{lemma}

\begin{proof}\leavevmode
  \begin{enumerate}
  \item Clear.
  \item Clear.
  \item The key step is replacing $k_1$ by $k_1 + 1$ in the following computation
    \begin{align*}
      f_{a, b, c, d}(t, q) - f_{a + n, b, c, d}(t, q) &= \sum_{k_1, k_2 \in \mathbb{N}}t^{4k_1 + 2k_2 + d}\frac{q^{4k_1^2 + 3k_1k_2 + k_2^2 + ak_1 + bk_2 + c}}{(q)_{k_1}(q)_{k_2}}(1 - q^{nk_1}) \\
                                                      &= \sum_{k_1, k_2 \in \mathbb{N}}t^{4k_1 + 2k_2 + d}\frac{q^{4k_1^2 + 3k_1k_2 + k_2^2 + ak_1 + bk_2 + c}}{(q)_{k_1}(q)_{k_2}}(1 - q^{k_1})\sum_{k = 0}^{n - 1}q^{kk_1} \\
                                                      &= \sum_{k = 0}^{n - 1}\sum_{k_1, k_2 \in \mathbb{N}}t^{4k_1 + 2k_2 + d}\frac{q^{4k_1^2 + 3k_1k_2 + k_2^2 + (a + k)k_1 + bk_2 + c}}{(q)_{k_1}(q)_{k_2}}(1 - q^{k_1}) \\
                                                      &= \sum_{k = 0}^{n - 1}\sum_{k_1, k_2 \in \mathbb{N}}t^{4k_1 + 2k_2 + d + 4}\frac{q^{4k_1^2 + 3k_1k_2 + k_2^2 + (a + 8 + k)k_1 + (b + 3)k_2 + a + c + 4 + k}}{(q)_{k_1}(q)_{k_2}} \\
                                                      &= \sum_{k = 0}^{n - 1}f_{a + 8 + k, b + 3, a + c + 4 + k, d + 4}(t, q).
    \end{align*}
  \item Same trick as (iii) but with $k_2$ instead of $k_1$. \qedhere
  \end{enumerate}
\end{proof}

\section{A combinatorial argument}
\label{sec:comb-argum}

A partition $\lambda = [\lambda_1, \dots, \lambda_m]$ \emph{contains a partition} $\eta = [\eta_1, \dots, \eta_n]$, written as $\eta \subseteq \lambda$, if $m \ge n$ and there is $i \in \mathbb{Z}_+$ such that $1 \le i \le m - n + 1$ and $[\lambda_i, \lambda_{i + 1}, \dots, \lambda_{i + n - 1}] = \eta$.

We define
\begin{equation*}
  p(t, q) = \sum_{\lambda \in P}t^{\len(\lambda)}q^{\Delta(\lambda)} \in \mathbb{C}[[t, q]],
\end{equation*}
where $P$ is the set of partitions that do not contain any partition in $R$ as defined in \zcref{sec:introduction}, i.e.,
\begin{equation*}
  P = \{\lambda \mid \text{for $\eta \in R$, $\lambda \nsupseteq \eta$}\}.
\end{equation*}
We call the last eight partitions of $R$ \emph{exceptional partitions}, and the others involving $r$ are called \emph{ordinary partitions}.
For $m, n \in \mathbb{N}$, we set:
\begin{align*}
  P(n) &= \{\lambda \in P \mid \Delta(\lambda) = n\}, \\
  p(q) &= \sum_{n \in \mathbb{N}}|P(n)|q^n \in \mathbb{C}[[q]], \\
  P(n, m) &= \{\lambda \in P \mid \text{$\len(\lambda) = m$ and $\Delta(\lambda) = n$}\}.
\end{align*}
Therefore, we have:
\begin{align*}
  p(t, q) &= \sum_{m, n \in \mathbb{N}}|P(n, m)|t^mq^n, \\
  p(1, q) &= p(q).
\end{align*}
We wish to find an expression for $p(t, q)$ as a sum of series $f_{a, b, c, d}(t, q)$ for some tuples $(a, b, c, d)$.

We now define subsets of $P$, which will help us in finding an expression for $p(t, q)$, by setting:
\begin{align*}
  P_{>2} &= \{[\lambda_1, \dots, \lambda_m] \in P \mid \text{$\lambda_m > 2$ or $\lambda = \emptyset$}\}, \\
  P_2 &= \{[\lambda_1, \dots, \lambda_m] \in P \mid \lambda_m = 2\},
\end{align*}
and both $p_{>2}(t, q)$ and $P_{>2}(n, m)$ are defined like $p(t, q)$ and $P(n, m)$ were defined.
Likewise, we define $P_{>6, 5, 3}$, $P_{6, 5, 3}$, $p_{>6, 5, 3}(t, q)$, $p_{6, 5, 3}(t, q)$, $P_{>6, 5, 3}(n, m)$ and $P_{6, 5, 3}(n, m)$.
It turns out that $P$ decomposes as a disjoint union of these smaller objects, and we can find recurrence relations between them to find our desired formula for $p(t, q)$.

\begin{lemma}
  \label{lmm:2}
  The formal power series $p_{>2}(t, q)$ is given by
  \begin{equation*}
    p_{>2}(t, q) = f_{3, 2, 0, 0}(t, q).
  \end{equation*}
\end{lemma}

\begin{proof}
  We consider the disjoint union
  \begin{equation*}
    P_{>2} = P_{4, 3} \cup P_{6, 5, 3} \cup P_{>6, 5, 3} \cup P_{>5, 3} \cup P_{4, 4} \cup P_{5, 4} \cup P_{>5, 4} \cup P_{>4},
  \end{equation*}
  from which we get the formula
  \begin{equation*}
    p_{>2}(t, q) = p_{4, 3}(t, q) + p_{6, 5, 3}(t, q) + p_{>6, 5, 3}(t, q) + p_{>5, 3}(t, q) + p_{4, 4}(t, q) + p_{5, 4}(t, q) + p_{>5, 4}(t, q) + p_{>4}(t, q).
  \end{equation*}
  These subseries satisfy the following recurrences with initial conditions:
  \begin{align*}
    p_{>4}(t, q) &= p_{4, 4}(tq^{1/2}, q) + p_{5, 4}(tq^{1/2}, q) + p_{>5, 4}(tq^{1/2}, q) + p_{>4}(tq^{1/2}, q), &p_{>4}(0, 0) &= 1, \\
    p_{>5, 4}(t, q) &= p_{6, 5, 3}(tq^{1/2}, q) + p_{>6, 5, 3}(tq^{1/2}, q) + p_{>5, 3}(tq^{1/2}, q), &p_{>5, 4}(0, 0) &= 0, \\
    p_{5, 4}(t, q) &= p_{4, 3}(tq^{1/2}, q), &p_{5, 4}(0, 0) &= 0, \\
    p_{4, 4}(t, q) &= t^2q^{3}p_{>6, 5, 3}(tq^{2/2}, q) + t^2q^3p_{>5, 3}(tq^{2/2}, q), &p_{4, 4}(0, 0) &= 0, \\
    p_{>5, 3}(t, q) &= t^2q^3p_{>4}(tq^{1/2}, q), &p_{>5, 3}(0, 0) &= 0, \\
    p_{>6, 5, 3}(t, q) &= t^2q^3p_{>5, 4}(tq^{1/2}, q), &p_{>6, 5, 3}(0, 0) &= 0, \\
    p_{6, 5, 3}(t, q) &= t^2q^3p_{5, 4}(tq^{1/2}, q), &p_{6, 5, 3}(0, 0) &= 0, \\
    p_{4, 3}(t, q) &= t^2q^3p_{4, 4}(tq^{1/2}, q) + t^2q^2p_{>5, 4}(tq^{1/2}, q), &p_{4, 3}(0, 0) &= 0.
  \end{align*}
  The solution to these equations is unique if it exists, and we can verify using \zcref{lmm:1} that:
  \begin{align*}
    p_{>4}(t, q) &= f_{6, 4, 0, 0}(t, q), &p_{>5, 4}(t, q) &= f_{9, 5, 4, 2}(t, q), \\
    p_{5, 4}(t, q) &= f_{13, 6, 9, 4}(t, q), &p_{4, 4}(t, q) &= f_{12, 6, 8, 4}(t, q), \\
    p_{>5, 3}(t, q) &= f_{8, 5, 3, 2}(t, q), &p_{>6, 5, 3}(t, q) &= f_{11, 6, 8, 4}(t, q), \\
    p_{6, 5, 3}(t, q) &= f_{15, 7, 14, 6}(t, q), &p_{4, 3}(t, q) &= f_{11, 5, 7, 4}(t, q),
  \end{align*}
  is a solution to these equations.
  We derive the formula $p_{>2}(t, q) = f_{3, 2, 0, 0}(t, q)$ again from \zcref{lmm:1}.

  The first recurrence follows from the following bijections for $m, n \in \mathbb{N}$.
  \begin{align*}
    P_{>4}(n, 2m) &\xrightarrow{\sim} P_{4, 4}(n - m, 2m) \cup P_{5, 4}(n - m, 2m) \cup P_{>5, 4}(n - m, 2m) \cup P_{>4}(n - m, 2m), \\
    \lambda &\mapsto
              \begin{cases}
                [\lambda_1 - 1, \dots, \lambda_{m - 2} - 1, 4, 4] &\text{if $[\lambda_{m - 1}, \lambda_m] = [5, 5]$}; \\
                [\lambda_1 - 1, \dots, \lambda_{m - 2} - 1, 5, 4] &\text{if $[\lambda_{m - 1}, \lambda_m] = [6, 5]$}; \\
                [\lambda_1 - 1, \dots, \lambda_{m - 1} - 1, 4] &\text{if $[\lambda_m] = [5]$ and $\lambda_{m - 1} > 6$}; \\
                [\lambda_1 - 1, \dots, \lambda_m - 1] &\text{if $\lambda_m > 5$},
              \end{cases}
  \end{align*}
  which can be verified directly from the definition of $P$.
  The other recurrences are proven in a similar way.
\end{proof}

\begin{lemma}
  \label{lmm:3}
  The formal power series $p_{>2, 1}(t, q)$ is given by
  \begin{equation*}
    p_{>2, 1}(t, q) = f_{5, 2, 1, 1}(t, q).
  \end{equation*}
\end{lemma}

\begin{proof}
  From the disjoint union
  \begin{equation*}
    P_{>2, 1} = P_{5, 3, 1} \cup P_{>5, 3, 1} \cup P_{5, 4, 1} \cup P_{>5, 4, 1} \cup P_{>4, 1},
  \end{equation*}
  the map $[\lambda_1, \dots, \lambda_m, 1] \mapsto [\lambda_1, \dots, \lambda_m]$ and the proof of \zcref{lmm:2}, we get:
  \begin{align*}
    p_{>4, 1}(t, q) &= tqp_{>4}(t, q) = f_{6, 4, 1, 1}(t, q), \\
    p_{>5, 4, 1}(t, q) &= tqp_{>5, 4}(t, q) = f_{9, 5, 5, 3}(t, q), \\
    p_{5, 4, 1}(t, q) &= tqp_{5, 4}(t, q) = f_{13, 6, 10, 5}(t, q), \\
    p_{>5, 3, 1}(t, q) &= tqp_{>5, 3}(t, q) = f_{8, 5, 4, 3}(t, q), \\
    p_{5, 3, 1}(t, q) &= tqp_{>6, 5, 3}(t, q) = f_{11, 6, 9, 5}(t, q).
  \end{align*}
  The formula for $p_{>2, 1}(t, q)$ then follows from \zcref{lmm:1}.
\end{proof}

\begin{lemma}
  \label{lmm:4}
  The formal power series $p_{>3, 1, 1}(t, q)$ is given by
  \begin{equation*}
    p_{>3, 1, 1}(t, q) = f_{6, 3, 2, 2}(t, q).
  \end{equation*}
\end{lemma}

\begin{proof}
  From the disjoint union
  \begin{equation*}
    P_{>3, 1, 1} = P_{4, 1, 1} \cup P_{>4, 1, 1},
  \end{equation*}
  the map $[\lambda_1, \dots, \lambda_m, 1] \mapsto [\lambda_1, \dots, \lambda_m]$ and the proof of \zcref{lmm:3}, we get:
  \begin{align*}
    p_{>4, 1, 1}(t, q) &= tqp_{>4, 1}(t, q) = f_{6, 4, 2, 2}(t, q), \\
    p_{4, 1, 1}(t, q) &= tqp_{>5, 4, 1}(t, q) = f_{9, 5, 6, 4}(t, q).
  \end{align*}
  The formula for $p_{>3, 1, 1}(t, q)$ then follows from \zcref{lmm:1}.
\end{proof}

\begin{lemma}
  \label{lmm:5}
  The formal power series $p(t, q)$ is given by
  \begin{equation*}
    p(t, q) = f_{3, 2, 0, 0}(t, q) + f_{5, 2, 1, 1}(t, q) + f_{6, 3, 2, 2}(t, q).
  \end{equation*}
\end{lemma}

\begin{proof}
  This follows from the disjoint union
  \begin{equation*}
    P = P_{>2} \cup P_{>2, 1} \cup P_{>3, 1, 1}
  \end{equation*}
  together with \zcref{lmm:2}, \zcref{lmm:3} and \zcref{lmm:4}.
\end{proof}

\begin{lemma}
  \label{lmm:6}
  The formal power series $p(q)$ satisfies
  \begin{equation*}
    \ch_{L}(q) = q^{1/2}p(q).
  \end{equation*}
\end{lemma}

\begin{proof}
  This follows from \zcref{lmm:5} by setting $t = 1$ and \cite[Theorem 4]{andrews_singular_2022} together with \zcref{lmm:1}.
\end{proof}

\section{Computing leading monomials}
\label{sec:comp-lead-monom}

We recall that we have defined $V = \Vir^{1/2}$, $M = M(1/2, 1/2)$ and $L = L(1/2, 1/2)$.

\begin{lemma}
  \label{lmm:7}
  Let $\lambda = [\lambda_1, \dots, \lambda_m]$ be a partition with $\lambda_m \ge 2$ or $\lambda = \emptyset$, and we consider $M$ as a module over $V$.
  Then
  \begin{equation*}
    (L_{\lambda}\vac)_{(-1)}\vachalf = L_{\lambda}\vachalf + u \quad \text{for some $u \in G^{2m - 1}M$}.
  \end{equation*}
\end{lemma}

\begin{proof}
  The isomorphism $\gr^G(V) \xrightarrow{\sim} \mathbb{C}[L_{-2}, L_{-3}, \dots]$ maps $L_{\lambda}\vac$ to $p_{\lambda}$.
  The isomorphism $\gr^G(M) \xrightarrow{\sim} \bigoplus_{k \in \mathbb{N}}\mathbb{C}[L_{-2}, L_{-3}, \dots]L_{-1}^k$ maps $\vachalf$ to $L_{-1}^0$ and $L_{\lambda}\vachalf$ to $u_{\lambda}$.
  The equality $p_{\lambda}\cdot L_{-1}^0 = u_{\lambda}$ in $\gr^G(M)$ translates to the equality $(L_{\lambda}\vac)_{(-1)}\vachalf = L_{\lambda}\vachalf + u$ in $M$ for some $u \in G^{2m - 1}M$.
\end{proof}

We know $\gr^G(M)$ is a free module over $\gr^G(V)$
\begin{equation*}
  \gr^G(M) = \bigoplus_{k \in \mathbb{N}}\gr^G(V)L_{-1}^k.
\end{equation*}
For $k \in \mathbb{N}$, we call
\begin{equation*}
  \iota_k: \gr^G(V) \hookrightarrow \gr^G(M)
\end{equation*}
the insertion of $\gr^G(V)$ into the $k$-th component of $\gr^G(M)$.

\begin{lemma}
  \label{lmm:8}
  Let $a \in G^pV$ for some $p \in \mathbb{N}$.
  Then $a_{(-1)}\vachalf \in G^pM$ and
  \begin{equation*}
    \iota_0(\alpha^p(a)) = \alpha^p(a_{(-1)}\vachalf).
  \end{equation*}
\end{lemma}

\begin{proof}
  The result follows immediately from \zcref{lmm:7}.
\end{proof}

We have a natural quotient map
\begin{align*}
  \pi_0: V &\twoheadrightarrow \Vir_{3, 4}, \\
  \pi_0(a) &= a + U(\Vir)\{a_{3, 4}\},
\end{align*}
where
\begin{equation*}
  a_{3, 4} = (L_{-2}^3 + \tfrac{93}{64}L_{-3}^2 - \tfrac{27}{16}L_{-6} - \tfrac{33}{8}L_{-4}L_{-2})\vac
\end{equation*}
is the singular vector of $V$ that generates its maximal proper ideal as in \cite{andrews_singular_2022}.
Applying the functor $\gr^G$, we obtain an epimorphism of $\mathbb{N}$-graded vertex Poisson algebras
\begin{equation*}
  \gr^G(\pi_0): \gr^G(V) \twoheadrightarrow \gr^G(\Vir_{3, 4}).
\end{equation*}
We set $I = \ker(\gr^G(\pi_0))$, following the notation of \cite{andrews_singular_2022}, and we recall the definition of $K$ given in \eqref{eq:5}.

\begin{lemma}
  \label{lmm:9}
  We have the inclusion
  \begin{equation*}
    \iota_0(I) \subseteq K.
  \end{equation*}
\end{lemma}

\begin{proof}
  We can consider $L$ as a module over $V$ with state-field correspondence map $Y^L_{V}: V \to \mathcal{F}(L)$.
  By \cite[Theorem 4.2]{wang_rationality_1993}, $L$ is a module over $\Vir_{3,4}$ with state-field correspondence map $Y^L_{\Vir_{3, 4}}: \Vir_{3, 4} \to \mathcal{F}(L)$ such that the following diagram commutes
  \begin{equation*}
    \begin{tikzcd}
      V \arrow[r, "\pi_0", two heads] \arrow[rd, "{Y^{L}_{V}}"'] & {\Vir_{3, 4}} \arrow[d, "{Y^{L}_{\Vir_{3, 4}}}"] \\
      & {\mathcal{F}(L)}
    \end{tikzcd}
  \end{equation*}
  The commutativity of this diagram implies the following statement
  \begin{equation*}
    \text{for $a \in U(\Vir)\{a_{3, 4}\}$, $u \in M$ and $n \in \mathbb{Z}$, $a_{(n)}u \in W$},
  \end{equation*}
  where $W$ is defined in \eqref{eq:3}.
  We use this statement in the following simplified form
  \begin{equation}
    \label{eq:6}
    \text{for $a \in U(\Vir)\{a_{3, 4}\}$, $a_{(-1)}\vachalf \in W$}.
  \end{equation}

  We note that
  \begin{equation*}
    I = \sum_{p \in \mathbb{N}}\alpha^p(U(\Vir)\{a_{3, 4}\} \cap G^pV).
  \end{equation*}
  Similarly, we have a formula for $K$
  \begin{equation*}
    K = \sum_{p \in \mathbb{N}}\alpha^p(W \cap G^pM).
  \end{equation*}
  Let $\alpha^p(a) \in I$ with $a \in U(\Vir)\{a_{3, 4}\} \cap G^pV$.
  By \zcref{lmm:8} and \eqref{eq:6}, we have $a_{(-1)}\vachalf \in W \cap G^pM$ and also $\iota_0(\alpha^p(a)) = \alpha^p(a_{(-1)}\vachalf) \in K$, finishing the proof.
\end{proof}

\begin{remark}
  \label{rmk:5}
  The proof of \zcref{lmm:9} also works for other values of $c = c_{p, q} = 1 - \frac{6(p - q)^2}{pq}$ and $h = h_{m, n} = \frac{(np - mq)^2 - (p - q)^2}{4pq}$, where $p, q \ge 2$ are relatively prime integers, and $m, n$ are integers satisfying $0 < m < p$ and $0 < n < q$, see \cite{wang_rationality_1993}.
\end{remark}

We need to compute all leading monomials of elements of $K$.
To do this, we need to order the PBW basis of $U(\Vir_{\le -1}) = \vspan\{L_{\lambda} \mid \text{$\lambda$ is a partition}\}$ first by length, then by degree reverse lexicographic order and finally by position.
Formally, for any partitions $\lambda$ and $\eta$, we define
\begin{equation*}
  L_{\lambda} \le L_{\eta}\text{ if and only if }
  \begin{cases}
    \len(\lambda) < \len(\eta)\text{ or } \\
    \len(\lambda) = \len(\eta)\text{ and }u_{\lambda} \le u_{\eta}.
  \end{cases}
\end{equation*}
For $x \in U(\Vir_{\le -1})$ with $x \neq 0$, we may write
\begin{equation*}
  x = c_1L_{\lambda_1} + c_2L_{\lambda_2} + \dots + c_rL_{\lambda_r},
\end{equation*}
where for $1 \le i \le r$, $0 \neq c_i \in \mathbb{C}$ and $L_{\lambda_1} > L_{\lambda_2} > \dots > L_{\lambda_r}$.
We define the \emph{leading monomial of $x$} as $\lm(x) = L_{\lambda_1}$.
We set $\lm(0) = 0$.
Next, we extend the definition of $\lm$ from $U(\Vir_{\le -1})$ to $M$ by considering the isomorphism of vector spaces $U(\Vir_{\le -1}) \xrightarrow{\sim} M, L_{\lambda} \mapsto L_{\lambda}\vachalf$, where $\lambda$ is a partition.
For example:
\begin{enumerate}
\item We have $\lm(L_{-1}^2 - \frac{3}{4}L_{-2}) = L_{-2}$ because the power product of $L_{-1}^2$ is $1$, the power product of $L_{-1}^0$ is $L_{-2}$, $L_{-2} > 1$ by degree considerations, and we are using TOP with $L_{-1}^0 < L_{-1}^1 < L_{-1}^2$.
  We also have $\lm((L_{-1}^2 - \frac{3}{4}L_{-2})\vachalf) = L_{-2}\vachalf$ because $\len(L_{-1}^2) = \len(L_{-2}) = 2$.
\item We have $\lm(L_{-1}^3 - 3L_{-2}L_{-1} + \frac{3}{4}L_{-3}) = L_{-2}L_{-1}$ because the power product of $L_{-1}^3$ is $1$, the power product of $L_{-1}^1$ is $L_{-2}$, the power product of $L_{-1}^0$ is $L_{-3}$, $L_{-2} > L_{-3} > 1$ by definition and degree considerations, and we are using TOP with $L_{-1}^0 < L_{-1}^1 < L_{-1}^2 < L_{-1}^3$.
  Again, we have $\lm((L_{-1}^3 - 3L_{-2}L_{-1} + \frac{3}{4}L_{-3})\vachalf) = L_{-2}L_{-1}\vachalf$ because $\len(L_{-1}^3) = \len(L_{-2}L_{-1}) = 3$ and $\len(L_{-3}) = 2$.
\item We have $\lm(L_{-1}^4 - 3L_{-3}L_{-1} - 6L_{-4}) = L_{-3}L_{-1}$ because the power product of $L_{-1}^4$ is $1$, the power product of $L_{-1}^1$ is $L_{-3}$, the power product of $L_{-1}^0$ is $L_{-4}$, $L_{-3} > L_{-4} > 1$ by definition and degree considerations, and we are using TOP with $L_{-1}^0 < L_{-1}^1 < L_{-1}^2 < L_{-1}^3 < L_{-1}^4$.
  However, we now have $\lm((L_{-1}^4 - 3L_{-3}L_{-1} - 6L_{-4})\vachalf) = L_{-1}^4\vachalf$ because $\len(L_{-1}^4) = 4$, $\len(L_{-3}L_{-1}) = 3$ and $\len(L_{-4}) = 2$.
  In fact, we have $(L_{-1}^4 - 3L_{-3}L_{-1} - 6L_{-4})\vachalf \in W$, as we will see in a moment.
\end{enumerate}

\begin{remark}
  \label{rmk:6}
  The definition of the order in the PBW basis of $U(\Vir_{\le -1})$ was made so that for a partition $\lambda$ and $u \in M$, if $\lm(u) = L_{\lambda}\vachalf$, then $\lm(\alpha^{\len(\lambda)}(u)) = u_{\lambda}$.

  This order also helps us in computing the exceptional partitions faster, as we do not have to compute Gröbner bases explicitly because it is enough to transform matrices into row reduced echelon form.
\end{remark}

For $n \in \mathbb{N}$, let $p(n)$ denote the number of partitions of $n$.
A basis of $\gr^G(M)_{1/2 + n}$ is given by $\{L_{\lambda}\vachalf \mid \Delta(\lambda) = n\}$ and has $p(n)$ elements.
For $n \in \mathbb{N}$, we compute the matrix $A_n$ with $p(n - 2) + p(n - 3)$ rows and $p(n)$ columns, which is given by stacking the matrix $A^{m_3}_n$ below the matrix $A^{m_2}_n$.
The matrix $A^{m_k}_n$ is given by
\begin{equation*}
  A^{m_k}_n(i, j) = \text{coefficient of $L_{\lambda_j}\vachalf$ in $L_{\lambda_i}u_k$},
\end{equation*}
for $1 \le i \le p(n - k)$, $1 \le j \le p(n)$ and $k = 2, 3$, where $\lambda_1, \dots, \lambda_{p(n)}$ are the partitions of $n$ ordered in such a way that $L_{\lambda_1} > L_{\lambda_2} > \dots > L_{\lambda_{p(n)}}$, and $u_2, u_3$ are given by \eqref{eq:4}.
We now transform $A_n$ into row reduced echelon form, obtaining a matrix $A^W_n$ which has an unknown number of nonzero rows and $p(n)$ columns.
For each pivot $\lambda$ of $A^W_n$, let $u^W_{\lambda}$ be the element of $\gr^G(M)$ corresponding to the row which has $L_{\lambda}\vachalf$ as pivot.
In other words, $u^W_{\lambda} = L_{\lambda}\vachalf + \text{(lower order terms)}$.
We set $u^K_{\lambda} = \alpha^{\len(\lambda)}(u^W_{\lambda})$ for each pivot $\lambda$ of $A^W_n$.

\begin{remark}
  \label{rmk:7}
  By construction and \zcref{rmk:6}, for each pivot $\lambda$ of $A^W_n$, we have $u^W_{\lambda} \in W$, $u^K_{\lambda} \in K$ and $\lm(u^K_{\lambda}) = u_{\lambda}$.
\end{remark}

For example, when $n = 4$, the partitions are ordered in the following way
\begin{equation*}
  [[2, 2], [2, 1, 1], [1, 1, 1, 1], [3, 1], [4]],
\end{equation*}
the matrix $A^W_4$ is given by
\begin{equation*}
  A^W_4=
  \begin{pmatrix}
    1 & 0 & 0 & -\frac{3}{16} & -\frac{15}{8} \\
    0 & 1 & 0 & -\frac{1}{4} & -\frac{5}{2} \\
    0 & 0 & 1 & -3 & -6
  \end{pmatrix},
\end{equation*}
and the pivots are $[2, 2]$, $[2, 1, 1]$ and $[1, 1, 1, 1]$.
Therefore:
\begin{align*}
  u^W_{[2, 2]} &= (L_{[2, 2]} - \tfrac{3}{16}L_{[3, 1]} - \tfrac{15}{8}L_{[4]})\vachalf, &u^K_{[2, 2]} &= L_{-2}L_{-2}, \\
  u^W_{[2, 1, 1]} &= (L_{[2, 1, 1]} - \tfrac{1}{4}L_{[3, 1]} - \tfrac{5}{2}L_{[4]})\vachalf, &u^K_{[2, 1, 1]} &= L_{-2}L_{-1}^2, \\
  u^W_{[1, 1, 1, 1]} &= (L_{[1, 1, 1, 1]} - 3L_{[3, 1]} - 6L_{[4]})\vachalf, &u^K_{[1, 1, 1, 1]} &= L_{-1}^4.
\end{align*}

The partitions $[2]$, $[1, 1, 1]$, $[3, 1, 1]$, $[3, 3]$, $[4, 3, 1]$, $[4, 4, 1]$, $[5, 4, 1, 1]$ and $[6, 5, 3, 1]$ are pivots of the matrices $A^W_2, A^W_3$, $A^W_5$, $A^W_6$, $A^W_8$, $A^W_9$, $A^W_{11}$ and $A^W_{15}$ respectively, see \cite[ising-modules.ipynb]{sagemath2}.
We set
\begin{equation*}
  K' = (u^K_{[2]}, u^K_{[1, 1, 1]}, u^K_{[3, 1, 1]}, u^K_{[3, 3]}, u^K_{[4, 3, 1]}, u^K_{[4, 4, 1]}, u^K_{[5, 4, 1, 1]}, u^K_{[6, 5, 3, 1]}, \iota_0(I))_{\psn},
\end{equation*}
where $\psn$ denotes the Poisson submodule generated by the given subset.
By \zcref{lmm:9}, \zcref{rmk:7} and the fact that $K$ is a vertex Poisson submodule (not just a submodule), we have
\begin{equation*}
  K' \subseteq K.
\end{equation*}

Let us consider the Poisson structure of $\gr^G(M)$ as a module over $\gr^G(V)$.
We have
\begin{align*}
  (L_{-2})_{(0)}(u_{\lambda}) &= \alpha^2(\omega)_{(0)}\alpha^{\len(\lambda)}(L_{\lambda}\vachalf) \\
                              &= \alpha^{\len(\lambda) + 1}(L_{-1}L_{\lambda}\vachalf) \\
                              &= u_{[\lambda, 1]},
\end{align*}
where $\lambda$ is any partition, and $[\lambda, 1]$ denotes the partition $\lambda$ with a $1$ appended at the end.

\begin{remark}
  \label{rmk:8}
  While the Poisson structure of $\gr^G(V)$ is trivial (i.e., is zero), the Poisson structure of $\gr^G(M)$ is not.
  In fact, if $u \in \gr^G(M)$, and $\lambda$ is a partition, then
  \begin{equation*}
    \lm(u) = u_{\lambda}\text{ implies }\lm((L_{-2})_{(0)}(u)) = u_{[\lambda, 1]}.
  \end{equation*}
\end{remark}

Let $\overline{R}$ be the set of partitions containing some partition of $R$, i.e.,
\begin{equation*}
  \overline{R} = \{\lambda \mid \text{there is some partition $\eta \in R$ such that $\lambda \supseteq \eta$}\}.
\end{equation*}

\begin{lemma}
  \label{lmm:10}
  For $\lambda \in \overline{R}$, there exists $u \in K'$ such that $\lm(u) = u_{\lambda}$.
\end{lemma}

\begin{proof}
  By the definition of $K'$ and \cite[Proposition 5.1]{andrews_singular_2022}, we know that for $\lambda \in R$, there exists $u \in K'$ such that $\lm(u) = u_{\lambda}$.
  We now assume $\lambda \in \overline{R}$, which means there is $\eta \in R$ such that $\lambda \supseteq \eta$.
  Therefore, $\lambda$ is obtained from $\eta$ by adding some integers greater than or equal to two and adding $k$ ones.
  We pick $v \in K'$ such that $\lm(v) = u_{\eta}$ and some power product $p_{\tau} \in \mathbb{C}[L_{-2}, L_{-3}, \dots]$, for some partition $\tau$, such that $\lm((L_{-2})^k_{(0)}(p_{\tau}v)) = u_{\lambda}$ (this can be done because of \zcref{rmk:8} and the fact that $\lm$ is multiplicative).
  Thus, we set $u = (L_{-2})^k_{(0)}(p_{\tau}v) \in K'$ to get $\lm(u) = u_{\lambda}$.
  For example, if we take $\lambda = [3, 2, 1, 1] \in \overline{R}$, then $\eta = [2] \in R$.
  In this case, $v = L_{-2}- 3/4L_{-1}^2 \in K'$ is such that $\lm(v) = u_{[2]}$, and we take $u = L_{-3}L_{-2}L_{-1}^2 - 3/4L_{-3}L_{-1}^4 \in K'$.
\end{proof}

\begin{remark}
  \label{rmk:9}
  For any partitions $\lambda$ and $\eta$ with the same number of ones, if $\lambda \supseteq \eta$, then $u_{\eta} \mid u_{\lambda}$.
  The converse is not true.
  For example, $u_{[4, 2]} \mid u_{[4, 3, 2]}$, but $[4, 3, 2] \nsupseteq [4, 2]$.
  However, if $\eta = [\eta_1, \dots, \eta_m, 1, \dots, 1]$ with $\eta_m \ge 2$ and $\eta_1 - \eta_m \le 1$, then $\lambda \supseteq \eta$ if and only if $u_{\eta} \mid u_{\lambda}$, provided $\lambda$ and $\eta$ have the same number of ones.
\end{remark}

\begin{lemma}
  \label{lmm:11}
  There is an alternative description for $P$, namely
  \begin{equation*}
    P = \{\lambda \mid \text{for $\eta \in \overline{R}$, $u_{\eta} \nmid u_{\lambda}$}\}.
  \end{equation*}
\end{lemma}

\begin{proof}
  We have to prove the following equality
  \begin{equation*}
    \{\lambda \mid \text{for $\eta \in R$, $\lambda \nsupseteq \eta$}\} = \{\lambda \mid \text{for $\eta \in \overline{R}$, $u_{\eta} \nmid u_{\lambda}$}\}.
  \end{equation*}

  First, we prove the inclusion $(\supseteq)$.
  We assume $\lambda$ belongs to the right set, and $\lambda \supseteq \eta$ for some $\eta \in R$.
  Then $\lambda$ has $k$ more ones than $\eta$ for some $k \in \mathbb{N}$.
  By \zcref{rmk:9}, $u_{[\eta, 1, \dots, 1]} \mid u_{\lambda}$, where we added $k$ ones to $\eta$, a contradiction.
  Therefore, $\lambda \nsupseteq \eta$ for $\eta \in R$, which means $\lambda$ belongs to the left set.

  We now prove the inclusion $(\subseteq)$.
  We assume $\lambda$ belongs to the left set.
  Then $\lambda$ has $0$, $1$ or $2$ ones.
  We assume $\lambda$ has $0$ ones.
  It is enough to prove that for $\eta \in R$ with $0$ ones, $u_{\eta} \nmid u_{\lambda}$.
  By \zcref{rmk:9}, if $\eta$ is equal to $[2]$, $[3, 3]$, $[r, r, r]$, $[r + 1, r, r]$ or $[r + 1, r + 1, r]$ for some $r \ge 3$, then $u_{\eta} \nmid u_{\lambda}$.
  We assume $\eta = [r + 2, r + 1, r]$ for some $r \ge 3$ and $u_{\eta} \mid u_{\lambda}$.
  Then $\lambda$ contains $[r + 2, r + 1, \dots, r + 1, r]$, where $r + 1$ appears $k \ge 1$ times.
  If $k \ge 2$, then $\lambda$ contains $[r + 1, r + 1, r]$, which is not possible.
  If $k = 1$, then $\lambda$ contains $[r + 2, r + 1, r]$, which is not possible.
  Therefore, we cannot have $u_{\eta} \mid u_{\lambda}$.
  Continuing this way, we obtain that $u_{\eta} \nmid u_{\lambda}$ for $\eta \in R$ with $0$ ones.
  The same argument can be applied when $\lambda$ has $1$ or $2$ ones.
  Therefore, $\lambda$ belongs to the right set.
\end{proof}

\section{Proofs of the main theorems}
\label{sec:proofs-main-theorems}

\begin{proof}[Proof of \zcref{thr:4}]
  Since we cannot apply Gröbner basis theory directly on the free module $\bigoplus_{k \in \mathbb{N}}\mathbb{C}[L_{-2}, L_{-3}, \dots]L_{-1}^k$, we need to truncate somehow.
  For $N \in \mathbb{N}$, we define:
  \begin{align*}
    \gr^G(M)_{\le 1/2 + N} &= \sum_{n \le N}\gr^G(M)_{1/2 + n}, \\
    \gr^G(L)_{\le 1/2 + N} &= \sum_{n \le N}\gr^G(L)_{1/2 + n}.
  \end{align*}
  We note that $\gr^G(M)_{\le 1/2 + N}$ is a vector subspace of the free module
  \begin{equation*}
    F_N = \bigoplus_{n \le N}\mathbb{C}[L_{-2}, L_{-3}, \dots, L_{-N}]L_{-1}^n
  \end{equation*}
  with base ring $\mathbb{C}[L_{-2}, L_{-3}, \dots, L_{-N}]$ because a basis of $\gr^G(M)_{\le 1/2 + N}$ is given by elements of the form $u_{\lambda}$, with all the elements of $\lambda$ being at most $N$, and with $\lambda$ having at most $N$ ones.
  We note that $K \cap F_N$ is a submodule of $F_N$ for $N \in \mathbb{N}$.

  For $N \in \mathbb{N}$, we have natural vector space isomorphisms
  \begin{equation}
    \label{eq:7}
    \frac{F_N}{K \cap F_N} \supseteq \frac{\gr^G(M)_{\le 1/2 + N}}{K \cap F_N} \xrightarrow{\sim} \frac{\gr^G(M)_{\le 1/2 + N}}{K \cap \gr^G(M)_{\le 1/2 + N}} \xrightarrow{\sim} \gr^G(L)_{\le 1/2 + N}.
  \end{equation}
  Therefore, if we find a basis of each vector space $\gr^G(M)_{\le 1/2 + N}/K \cap F_N \subseteq F_N/K \cap F_N$ such that each basis is contained in the next one when considering the isomorphism \eqref{eq:7}, we get a basis of $\gr^G(L)$ by taking the union of these bases because $\bigcup_{N \in \mathbb{N}}\gr^G(L)_{\le 1/2 + N} = \gr^G(L)$.

  Let $G_N$ be a Gröbner basis of $K \cap F_N$.
  We define
  \begin{equation*}
    B_N = \{u_{\lambda} \mid \text{$\Delta(\lambda) \le N$ and for $u \in G_N$, $\lm(u) \nmid u_{\lambda}$}\}.
  \end{equation*}
  By \cite[Proposition 3.6.4]{adams_introduction_1994},
  \begin{equation}
    \label{eq:8}
    \{u_{\lambda} + K \cap F_N \mid u_{\lambda} \in B_N\}
  \end{equation}
  is a vector space basis of $\gr^G(M)_{\le 1/2 + N}/K \cap F_N$.
  Therefore, by \zcref{lmm:6}, isomorphism \eqref{eq:7} and the equality $\dim(\gr^G(L)_{1/2 + n}) = \dim(L_{1/2 + n})$ for $n \in \mathbb{N}$, we have
  \begin{equation*}
    |B_N| = \sum_{n \le N}\dim(\gr^G(L)_{1/2 + n}) = \sum_{n \le N}\dim(L_{1/2 + n}) = \sum_{n \le N}|P(n)|.
  \end{equation*}

  We define
  \begin{equation*}
    \overline{B_N} = \{u_{\lambda} \mid \text{$\Delta(\lambda) \le N$ and for $\eta \in \overline{R}$, $u_{\eta} \nmid u_{\lambda}$}\}.
  \end{equation*}
  By \zcref{lmm:11}, we also have
  \begin{equation*}
    \overline{B_N} = \{u_{\lambda} \mid \text{$\Delta(\lambda) \le N$ and $\lambda \in P$}\}.
  \end{equation*}
  From the definition of $P(n)$, we see that
  \begin{equation*}
    |\overline{B_N}| = \sum_{n \le N}|P(n)|.
  \end{equation*}

  On the other hand, $B_N \subseteq \overline{B_N}$.
  This is because if $u_{\lambda} \in B_N$ and $\eta \in \overline{R}$, then $u_{\eta} \nmid u_{\lambda}$, as we now show.
  Suppose, for the sake of contradiction, that $u_{\eta} \mid u_{\lambda}$.
  By \zcref{lmm:10}, there would exist $u \in K' \cap F_N \subseteq K \cap F_N$ such that $\lm(u) = u_{\eta}$.
  However, by the definition of a Gröbner basis, there would then exist $v \in G_N$ such that $\lm(v) \mid \lm(u)$, implying $\lm(v) \mid u_{\lambda}$, a contradiction.
  As both $B_N$ and $\overline{B_N}$ are finite sets and $|B_N| = |\overline{B_N}|$, we get $B_N = \overline{B_N}$.
  We see that the family of bases given by \eqref{eq:8} satisfies the property that each basis is a subset of the next.
  This observation and \zcref{prp:3} conclude the proof of \zcref{thr:4}.
\end{proof}

\begin{proof}[Proof of \zcref{thr:3}]
  We recall that the basis of \zcref{thr:4} also gives a basis of $\gr^G(L)$.
  For a partition $\lambda$, we have
  \begin{equation*}
    \alpha^{\len(\lambda)}(L_{\lambda}(\vachalf + W)) \in \alpha^{\len(\lambda)}(G^{\len(\lambda)}L_{1/2 + \Delta(\lambda)}).
  \end{equation*}
  Thus, when calculating the refined character of $\gr^G(L)$, $\Delta(\lambda)$ is incremented by $1/2$.
  Therefore, $\ch_{\gr^G(L)}(t, q) = q^{1/2}p(t, q)$ and together with \zcref{lmm:5}, we conclude the proof of \zcref{thr:3}.
\end{proof}

\begin{corollary}
  \label{crl:1}
  We have the equality
  \begin{equation*}
    K' = K.
  \end{equation*}
\end{corollary}

\begin{proof}
  This proof is a copy of the proof of \zcref{thr:3}.
  We know that $K' \subseteq K$.

  Let $N \in \mathbb{N}$, and let $G'_N$ be a Gröbner basis of $K' \cap F_N$ considered as a submodule of $F_N$.
  Like in the proof of \zcref{thr:3}, we set
  \begin{equation*}
    B'_N = \{u_{\lambda} \mid \text{$\Delta(\lambda) \le N$ and for $u \in G'_N$, $\lm(u) \nmid u_{\lambda}$}\}.
  \end{equation*}
  By \cite[Proposition 3.6.4]{adams_introduction_1994}, we have
  \begin{equation*}
    |B'_N| = \dim\left(\frac{\gr^G(M)_{\le 1/2 + N}}{K' \cap F_N}\right) = \dim\left(\frac{\gr^G(M)_{\le 1/2 + N}}{K' \cap \gr^G(M)_{\le 1/2 + N}}\right).
  \end{equation*}

  We can apply the same argument as in the proof of \zcref{thr:3} and conclude that $B'_N \subseteq \overline{B_N}$, where $\overline{B_N}$ is the defined in the same way.
  But $B_N = \overline{B_N}$ and
  \begin{equation*}
    |B_N| = \dim\left(\frac{\gr^G(M)_{\le 1/2 + N}}{K \cap \gr^G(M)_{\le 1/2 + N}}\right).
  \end{equation*}
  Therefore, we have $|B'_N| \le |B_N|$, and this implies
  \begin{equation*}
    \dim(K' \cap \gr^G(M)_{\le 1/2 + N}) \ge \dim(K \cap \gr^G(M)_{\le 1/2 + N}) \ge \dim(K' \cap \gr^G(M)_{\le 1/2 + N}).
  \end{equation*}
  Consequently, $K' \cap \gr^G(M)_{\le 1/2 + N} = K \cap \gr^G(M)_{\le 1/2 + N}$ for $N \in \mathbb{N}$.
  Taking the union of these subspaces, we get $K' = K$.
\end{proof}

\appendix
\section{An explicit description of $K$}
\label{sec:an-expl-descr}

In \cite[ising-modules.ipynb]{sagemath2}, I wrote a SageMath program to compute $u^W_{\lambda}$ for all partitions $\lambda$ (if $\lambda$ is not a pivot, it returns $0$).
From its output, we obtain:
\begin{align*}
  u^K_{[2]} &= L_{-2} - \tfrac{3}{4}L_{-1}^2, &u^K_{[1, 1, 1]} &= L_{-1}^3, \\
  u^K_{[3, 1, 1]} &= L_{-3}L_{-1}^2, &u^K_{[3, 3]} &= L_{-3}L_{-3} + \tfrac{1}{3}L_{-4}L_{-1}^2, \\
  u^K_{[4, 3, 1]} &= L_{-4}L_{-3}L_{-1}, &u^K_{[4, 4, 1]} &= L_{-4}L_{-4}L_{-1} + \tfrac{9}{8}L_{-5}L_{-3}L_{-1}, \\
  u^K_{[5, 4, 1, 1]} &= L_{-5}L_{-4}L_{-1}^2, &u^K_{[6, 5, 3, 1]} &= L_{-6}L_{-5}L_{-3}L_{-1}.
\end{align*}

However, it turns out that $u^K_{[5, 4, 1, 1]}$ and $u^K_{[6, 5, 3, 1]}$ are redundant, as can be seen in \cite[m11-m15.ipynb]{sagemath2}.
From \cite[Theorem 2]{andrews_singular_2022}, we obtain the following explicit expression for $K$
\begin{equation*}
  K = (u^K_{[2]}, u^K_{[1, 1, 1]}, u^K_{[3, 1, 1]}, u^K_{[3, 3]}, u^K_{[4, 3, 1]}, u^K_{[4, 4, 1]}, \iota_0((a, b)_{\partial}))_{\psn},
\end{equation*}
where
\begin{align*}
  a = L_{-2}^3, b = L_{-4}L_{-3}L_{-2} + \tfrac{1}{6}L_{-5}L_{-2}^2.
\end{align*}

\section{The case $L(1/2, 1/16)$}
\label{sec:case-l12-116}

We can do the same analysis we did for $L(1/2, 1/2)$ but with $L(1/2, 1/16)$ instead.
The arguments are the same, but the computations are, of course, different.
These computations will be briefly shown now.

We recall the definitions of $P^{1/16}$ and $R^{1/16}$ given in \zcref{sec:introduction}.
For the definition of the series $p^{1/16}(t, q)$, the set $P^{1/16}(n, m)$ and other related notation, see \zcref{sec:comb-argum}.
Again, we omit superscripts.

The maximal proper subrepresentation of $M(1/2, 1/16)$ is generated by the singular vectors:
\begin{equation*}
  u_2 = (L_{-2} - \tfrac{4}{3}L_{-1}^2)\vacsixteen, u_4 = (L_{-2}L_{-2} - \tfrac{600}{49}L_{-2}L_{-1}^2 + \tfrac{144}{49}L_{-1}^4 + \tfrac{264}{49}L_{-3}L_{-1} - \tfrac{36}{49}L_{-4})\vacsixteen.
\end{equation*}

\begin{lemma}
  \label{lmm:12}
  We have:
  \begin{align*}
    p_{>2}(t, q) &= f_{2, 2, 0, 0}(t, q), &p_{>2, 1}(t, q) &= f_{4, 2, 1, 1}(t, q), \\
    p_{>2, 1, 1}(t, q) &= f_{9, 4, 5, 4}(t, q) + f_{5, 3, 2, 2}(t, q), &p_{>3, 1, 1, 1}(t, q) &= f_{7, 3, 3, 3}(t, q).
  \end{align*}
\end{lemma}

\begin{proof}[Sketch of proof]
  We use the same strategy as in \zcref{sec:comb-argum}.
\end{proof}

\begin{lemma}
  \label{lmm:13}
  The formal power series $p(t, q)$ is given by
  \begin{equation*}
    p(t, q) = f_{1, 1, 0, 0}(t, q) + f_{4, 2, 1, 1}(t, q) + f_{7, 3, 3, 3}(t, q).
  \end{equation*}
\end{lemma}

\begin{proof}
  This follows from the disjoint union
  \begin{equation*}
    P = P_{>2} \cup P_{>2, 1} \cup P_{>2, 1, 1} \cup P_{>3, 1, 1, 1}
  \end{equation*}
  together with \zcref{lmm:12} and \zcref{lmm:1}.
\end{proof}

The partitions $[2]$, $[1, 1, 1, 1]$, $[3, 1, 1, 1]$, $[3, 3, 1]$, $[4, 3, 1]$, $[4, 4, 1, 1]$, $[5, 4, 1, 1, 1]$, $[5, 5, 1, 1, 1]$, $[6, 5, 3, 1, 1]$, $[6, 6, 3, 1, 1]$, $[7, 6, 4, 1, 1, 1]$ and $[8, 7, 5, 3, 1, 1]$ are pivots of the matrices $A^W_2$, $A^W_4$, $A^W_6$, $A^W_7$, $A^W_8$, $A^W_{10}$, $A^W_{12}$, $A^W_{13}$, $A^W_{16}$, $A^W_{17}$, $A^W_{20}$ and $A^W_{25}$ respectively, see \cite[ising-modules-1-16.ipynb]{sagemath2}.

\bibliographystyle{alpha}
\bibliography{ising-modules}

\end{document}